\documentclass[11pt]{amsart}
\usepackage{amsmath,amssymb} 
\usepackage{graphicx}
\usepackage[utf8]{inputenc} 
\usepackage[font=small,labelfont=bf]{caption}

\begin{document}

\newtheorem{thm}{Theorem}[section]
\newtheorem{lem}[thm]{Lemma}
\newtheorem{prop}[thm]{Proposition}
\newtheorem{coro}[thm]{Corollary}
\newtheorem{defn}[thm]{Definition}
\newtheorem*{remark}{Remark}

\numberwithin{equation}{section}

\newcommand{\Z}{{\mathbb Z}} %cph changed from \mathbf
\newcommand{\Q}{{\mathbb Q}}
\newcommand{\PP}{{\mathbb P}}
\newcommand{\R}{{\mathbb R}}
\newcommand{\C}{{\mathbb C}}
\newcommand{\N}{{\mathbb N}}
\newcommand{\FF}{{\mathbb F}}
\newcommand{\T}{{\mathbb T}}
\newcommand{\fq}{\mathbb{F}_q}
\newcommand{\IS}{{\mathbb S}}

\newcommand{\fixmehidden}[1]{}

\def\scrA{{\mathcal A}}
\def\cB{{\mathcal B}}
\def\Eps{{\mathcal E}}
\def\cI{{\mathcal I}}
\def\scrD{{\mathcal D}}
\def\cF{{\mathcal F}}
\def\cL{{\mathcal L}}
\def\cM{{\mathcal M}}
\def\cN{{\mathcal N}}
\def\cP{{\mathcal P}}
\def\scrC{{\mathcal C}}
\def\scrR{{\mathcal R}}
\def\scrS{{\mathcal S}}

\newcommand{\rmk}[1]{\footnote{{\bf Comment:} #1}}

\renewcommand{\mod}{\;\operatorname{mod}}
\newcommand{\ord}{\operatorname{ord}}
\newcommand{\TT}{\mathbb{T}}
\renewcommand{\i}{{\mathrm{i}}}
\renewcommand{\d}{{\mathrm{d}}}
\renewcommand{\^}{\widehat}
\newcommand{\HH}{\mathbb H}
\newcommand{\Vol}{\operatorname{vol}}
\newcommand{\area}{\operatorname{area}}
\newcommand{\tr}{\operatorname{tr}}
\newcommand{\norm}{\mathcal N} % norm =(\frac{ n+\sqrt{n^2-4}} 2)^2
\newcommand{\intinf}{\int_{-\infty}^\infty}
\newcommand{\ave}[1]{\left\langle#1\right\rangle} %  average
\newcommand{\E}{\mathbb E}
\newcommand{\Var}{\operatorname{Var}}
\newcommand{\Cov}{\operatorname{Cov}}
\newcommand{\Prob}{\operatorname{Prob}}
\newcommand{\sym}{\operatorname{Sym}}
\newcommand{\disc}{\operatorname{disc}}
\newcommand{\CA}{{\mathcal C}_A}
\newcommand{\cond}{\operatorname{cond}} % conductor
\newcommand{\lcm}{\operatorname{lcm}}
\newcommand{\Kl}{\operatorname{Kl}} %Kloosterman sum
\newcommand{\leg}[2]{\left( \frac{#1}{#2} \right)}  % Legendre symbol
\newcommand{\id}{\operatorname{id}}
\newcommand{\beq}{\begin{equation}}
\newcommand{\eeq}{\end{equation}}
\newcommand{\bsp}{\begin{split}}
\newcommand{\esp}{\end{split}}
\newcommand{\bra}{\left\langle}
\newcommand{\ket}{\right\rangle}
\newcommand{\diam}{\operatorname{diam}}
\newcommand{\supp}{\operatorname{supp}}
\newcommand{\dist}{\operatorname{dist}}
\newcommand{\sgn}{\operatorname{sgn}}
\newcommand{\inte}{\operatorname{int}}
\newcommand{\ind}{\operatorname{ind}}
\newcommand{\Spec}{\operatorname{Spec}}
\newcommand{\sumstar}{\sideset \and^{*} \to \sum}

\newcommand{\LL}{\mathcal L} %L-function of u
\newcommand{\sumf}{\sum^\flat}
\newcommand{\Hgev}{\mathcal H_{2g+2,q}}
\newcommand{\USp}{\operatorname{USp}}
\newcommand{\conv}{*}
\newcommand{\CF}{c_0} % Fejer constant
\newcommand{\kerp}{\mathcal K}

\newcommand{\gp}{\operatorname{gp}}
\newcommand{\Area}{\operatorname{Area}}

\title[A Liouville-type theorem] 
{A Liouville-type theorem for Schr\"odinger equations with nonnegative potential} 
  
\author{Henrik Uebersch\"ar}
\address{Sorbonne Universit\'e and Universit\'e Paris Cit\'e, CNRS, IMJ-PRG, F-75005 Paris, France.}
\email{henrik.ueberschar@imj-prg.fr}
\date{\today}

\date{\today}
\maketitle

\begin{abstract}
Let $u$ be a solution  of $\Delta u=Vu$ on $\R^d$, where $V\in C^0(\R^d)$ be nonnegative and bounded. We prove that the condition
$$\int_{r_j\leq|x|\leq r_j+1}|u(x)|^2dx\to 0,$$ along any sequence $(r_j)$, $r_j\nearrow+\infty$, implies $u\equiv 0$ on $\R^d$. In particular, this implies the Landis
conjecture for solutions satisfying a sufficiently fast algebraic decay.
These results are generalized to exterior domains as well as for a class of nonlinear Schr\"odinger equations under suitable conditions on the zero set of the potential.
\end{abstract}

\section{Introduction} 
Consider a solution of the Schr\"odinger equation
\beq\label{Schr}
(-\Delta+V)u=0
\eeq
on $\R^d$, where $V$ is a bounded, nonnegative potential and $\Delta=\sum_{j=1}^d \partial_{x_i}^2$ denotes the Laplacian. 
We will consider both the linear case ($V=V(x)$), as well as the non-linear case ($V=V(x,u)$).

A general question in the theory of unique continuation at infinity of solutions of linear PDE is the following: Assuming that there is a solution $u$ of \eqref{Schr} on $\R^d$, under which assumptions on the decay of $|u(x)|$, as $|x|\to+\infty$, may we conclude that $u\equiv 0$ on $\R^d$? This question is closely related to Landis' conjecture which was first proposed by E. M. Landis in the 1950s \cite{Lan53,Lan71} and which asserts that a sufficiently fast exponential decay should be sufficient to imply triviality. 

The special case of nonnegative bounded potentials is of interest in the study of disordered quantum systems and has attracted considerable attention since the work of Bourgain-Kenig on Anderson localization in the Anderson-Bernoulli model in 2005 \cite{BouKen05,KSW15,DKW19,SiSou21,PiDa24,PiDaK25}. We note that, due to nonnegativity of the operator $-\Delta+V$, there should be no bound states, hence it is natural to expect an algebraic decay rather than the exponential decay threshold in the sign-changing case.

In 2015, Kenig-Silvestre-Wang proved a quantitative form of the Landis conjecture for nonnegative potentials in $d=2$ in the sense that, for any nontrivial solution with $u(0)=1$, satisfying an exponential growth condition, they gave a lower bound : $$\inf_{|x_0|=R}\sup_{|x-x_0|\leq 1}|u(x)|\geq \exp(-cR\log R)$$ for $R$ large enough. 

In 2020, Logunov-Malinnikova-Nadirashvili-Nazarov \cite{LMNN24} proved a quantitative lower bound for nontrivial solutions of order $\exp(-cR(\log R)^{3/2})$ for general bounded potentials in $d=2$ which was subsequently generalized for variable coefficients in \cite{LBSou23}. We also note that the conjecture is false for complex-valued potentials, where Meshkov \cite{Me92} showed that nontrivial solutions exist which satisfy a decay of order $e^{-c|x|^{4/3}}$.

\section{Statement of results}
We prove the following result in the linear case.
\begin{thm}\label{alg-cont}
Let $V\in C^0(\R^d)$ be bounded and nonnegative. Let $u$ be a solution of $\Delta u=Vu$ on $\R^d$. Assume that there exists a sequence $(r_j)_{j=0}^{+\infty}\nearrow+\infty$, s. t.
\beq\label{crit}
\int_{r_j\leq |x|\leq r_j+1}|u|^2\to 0, \quad \text{as}\; j\to+\infty.
\eeq
Then $u\equiv 0$ on $\R^d$.
\end{thm}
\begin{remark}
We point out the following observations:
\begin{itemize}
\item[(i)]{This result implies Landis' conjecture for continuous nonnegative potentials in any dimension: 
%for any solution $u$ with decay $|u(x)|\leq e^{-c|x|}$ for some $c>0$ the conditions of the theorem are  satisfied which implies $u\equiv0$ on $\R^d$. 
what is noteworthy is that a sufficiently fast {\bf algebraic decay} of the solutions along {\bf any sequence of annuli} of constant width is sufficient in this case. 

For comparison: the conditions in Thm 1.1. in \cite{PiDa24} require a global decay $O(|x|^{(2-d)/2})$, $d\geq 2$, as well as a super-exponential decay $|u(x_n)|\leq Ce^{-|x_n|}|x_n|^{(1-d)/2}$ along a sequence $(x_n)$, s. t. $|x_n|\to+\infty$.}\\

\item[(ii)]{{\bf No global decay condition} for  the solution $u$ on $\R^d$ is needed at all. This is very different from the general case of the conjecture (bounded, possibly sign-changing potentials) where the threshold is an exponential decay: $|u(x)|\leq e^{-c|x|}$ for any $x\in\R^d$ and $c>0$ large enough.}\\

\item[(iii)]{{\bf Exterior domains:} An analogue of the above result may also be derived for exterior domains $\Omega\subset\R^d$, where $\Omega$ is an unbounded subset of $\R^d$ with piecewise smooth boundary. We refer the reader to section \ref{sec-ext} for a sketch of the proof.
}
\end{itemize}

\end{remark}

We have the following corollary of Theorem \ref{alg-cont} which gives a quantitative algebraic lower bound for a nontrivial solution. 
\begin{coro}\label{alg-quant}
Let $u$ be a nontrivial solution of $\Delta u=Vu$ on $\R^d$, where $V\in C^0(\R^d)$ and $V\geq0$. Then there exists an absolute constant $c>0$ s. t. for any $R\in\N$ 
$$
\int_{R\leq |x|\leq R+1} |u(x)|^2\geq c.
$$
\end{coro}
\begin{proof}
See section \ref{sec-coro}.
\end{proof}

In the non-linear case we have the following result.
\begin{thm}\label{alg-nonl}
Let $V\in C^0(\R)$ be nonnegative and assume that its zero set $Z_V:=\{x\in\R \mid V(x)=0\}$ is an isolated point set. Let $u$ be a real-valued, bounded solution of $\Delta u=V(u)u$. If $u$ satisfies  criterion \eqref{crit}, then $u\equiv 0$ on $\R^d$.
\end{thm}

{\bf Acknowledgements:} The author would like to thank N. Lerner for many fruitful discussions on the subject of Liouville theorems for solutions of linear and non-linear elliptic PDE, as well as unique continuation theorems for solutions of elliptic PDE. Moreover, the author thanks both N. Lerner and S. Nonnenmacher for many helpful comments on an earlier draft of this article.

\section{Proofs of Theorems \ref{alg-cont} and \ref{alg-nonl}}\label{sec-pos}

Let $V:\R^d\times\R\mapsto\R$ and assume $V$ is bounded, nonnegative and continuous w.r.t. $x\in\R^d$ and $u\in\R$. Let us consider a solution $u$ to the non-linear partial differential equation
 \beq\label{pde}
 (-\Delta+V(\cdot,u))u=0,
 \eeq
 where we denote by $\Delta=\sum\partial_{x_i}^2$ the Laplacian on $\R^d$. 
 
 We have the following proposition which will be key in the proof.
 \begin{prop}\label{main-prop}
 Assume that a real-valued\footnote{We assume real-valuedness here, because the nonnegative potential $V(x,u)$ is only defined for $u\in\R$.} solution $u$ of \eqref{pde} satisfies the following condition: there exists $(r_j)_{j=0}^{+\infty}\subset\R_+$, $r_j\nearrow +\infty$, s. t.
 \beq
 \lim_{j\to+\infty}\int_{r_j\leq|x|\leq r_j+1}u^2 = 0.
 \eeq
Then we have for any ball $B(x_0,R)\subset\R^d$ that
 \beq
 \int_{B(x_0,R)} V(x,u(x))u(x)^2dx=0.
 \eeq
 \end{prop}
 
%We first consider the case when $V=V(x,u)\geq 0$ and in addition we assume that $V\in L_x^\infty C_u^0(\R^d,\R)$. 
 
\begin{proof}
Let $r\geq1$. We introduce a smooth indicator function $0\leq \chi \leq 1$, where $\chi=1$ on $B(0,r+3/8)$ and $\chi=0$ on $\R^d\setminus B(0,r+5/8)$. In addition, we assume that there exist absolute constants $k_1,k_2>0$ s. t. $|\nabla\chi|\leq k_1$, $|\Delta\chi|\leq k_2$ for any $r\geq 1$ (an explicit construction and computation of $k_1,k_2$ is given in subsection \ref{subsec-cutoff}).
 
 We introduce the smooth cutoff $v=\chi u$ which satisfies (due to compact support, and we note that by elliptic regularity we have $u\in C^2(\R^d)$)
 \begin{equation}
0\leq \int_{\R^d}|\nabla v|^2 = -\int_{\R^d} v\Delta v
 \end{equation}
 and from the identity
 $$\Delta v=(\Delta\chi)u+2\nabla \chi \cdot \nabla u +\chi \Delta u$$
 we get the inequality
 \beq\label{main-ineq}
 \int_{\R^d}Vv^2\leq -\int_{\R^d}(\Delta\chi)u v-2\int_{\R^d}(\nabla \chi \cdot \nabla u) v
 \eeq
 where we used $\Delta u=Vu$ and moved the integral on the l.h.s.
 
Let $$A_{r,w}:=\{x\in\R^d \mid r-w\leq|x|\leq r+w\}.$$ We estimate the first term on the r.h.s. by (recall $0\leq\chi\leq1$)
 \beq\label{int-est-1}
 k_2\int_{A_{r+1/2,1/8}}u^2
 \eeq
 and the second term by
 \beq\label{int-est-2}
 k_1\int_{A_{r+1/2,1/8}}|\nabla u||u|\leq k_1\left(\int_{A_{r+1/2,1/8}}u^2\right)^{1/2}\left(\int_{A_{r+1/2,1/8}}|\nabla u|^2\right)^{1/2}.
 \eeq

We will make use of the following inequality which is due to Caccioppoli and whose proof we include in subsection \ref{subsec-cacc}.
\begin{lem}\label{Cacc}
Let $u$ be a solution to \eqref{pde}. If $V(x,u)$ depends on $u$, then assume $|u(x)|\leq C$ for any $x\in\R^d$, $V(\cdot,u)\in L^\infty(\R^d)$ and that $V$ is continuous with respect to $u$. If $V(x,u)=V(x)$, then we assume only $V\in L^\infty(\R^d)$. 

Under these assumptions we have for any $x_0\in\R^d$
\beq\label{local-L2}
\int_{B(x_0,1/4)}|\nabla u|^2 \leq C(V)\int_{B(x_0,1/2)}u^2,
\eeq
where $C(V)$ denotes a positive constant which depends only on the potential $V$.
\end{lem}

We may choose a finite set of points $\scrC_r\subset\IS^{d-1}_{r+1/2}$  s. t.
\beq
A_{r+1/2,1/8}\subset \bigcup_{x_i\in\scrC} B(x_i,1/4) 
\eeq
and there exists an absolute constants $C_d>0$ such that any ball $B(x_i,1/2)$ with $x_i\in\scrC_r$ overlaps only with at most $C_d$ other balls $B(x_j,1/2)$ with $x_j\in\scrC_r$. In particular, this constant does not depend on the radius $r$.

Hence we have, by the above inclusion and an application of Lemma \ref{Cacc},
\beq\label{est0}
\int_{A_{r+1/2,1/8}}|\nabla u|^2 \leq \sum_{x_i\in\scrC_r}\int_{B(x_i,1/4)}|\nabla u|^2 \leq C(V)\sum_{x_i\in\scrC_r}\int_{B(x_i,1/2)} u^2.
\eeq

We estimate the sum on the r.h.s. as
\beq\label{est1}
\begin{split}
\sum_{x_i\in\scrC}\int_{B(x_i,1/2)}u^2 &=\sum_{x_i\in\scrC_r}\int_{\R^d}u(x)^2 \ind(x;B(x_i,1/2))dx\\
&= \int_{\R^d}u(x)^2 h_r(x)dx
\end{split}
\eeq
where
$$h_r(x):= \sum_{x_i\in\scrC_r} \ind(x;B(x_i,1/2))$$
and as $\ind(\cdot,A):\R^d\mapsto\{1,0\}$ we denote the indicator function on a set $A\subset\R^d$:
\beq
\begin{split}
\ind(x,A):=
\begin{cases}
1, \;\text{if}\; x\in A,\\
\\
0, \;\text{otherwise.}
\end{cases}
\end{split}
\eeq

By construction, we have $h_r\leq C_d$ and, in view of \eqref{est1}, this yields
\beq
\sum_{x_i\in\scrC}\int_{B(x_i,1/2)}u^2\leq C_d\int_{\supp h_r}u^2
\eeq
Now, we have the trivial inclusion $$\supp h_r=\bigcup_{x_\i\in\scrC_r}B(x_i,1/2)\subset A_{r+1/2,1/2}$$
which yields
\beq\label{est2}
\sum_{x_i\in\scrC}\int_{B(x_i,1/2)}u^2\leq C_d\int_{A_{r+1/2,1/2}}u^2
\eeq

%for any $f\in L^2(A_{r+1/2,1/2})$
%\beq
%\int_{A_{r+1/2,1/8}}|f|^2 \leq c_1\sum_{x_i\in\scrC}\int_{B(x_i,1/4)}|f|^2
%\eeq
%and
%\beq
%\sum_{x_i\in\scrC}\int_{B(x_i,1/2)}|f|^2 \leq c_2 \int_{A_{r+1/2,1/2}}|f|^2.
%\eeq

By combining the estimates \eqref{est0} and \eqref{est2} we get
\beq
\int_{A_{r+1/2,1/8}}|\nabla u|^2dx \leq C(V)C_d \int_{r\leq |x|\leq r+1}u^2.
\eeq
 
Let $B=B(x_0,R)$ a fixed ball in $\R^d$. In view of the inequalities \eqref{main-ineq}, \eqref{int-est-1} and \eqref{int-est-2}, we have for any $r$ large enough such that $B\subset B(0,r+3/8)$ the following inequality:
\beq
0\leq \int_B V u^2\leq \int_{\R^d}V\chi^2 u^2\leq C'\int_{r\leq |x|\leq r+1}u^2
\eeq
where we set $C'=k_1C(V)^{1/2}C_d^{1/2}+k_2$.

In particular, this holds for any $r=r_j$ for $j$ large enough.
If we now take $j\to+\infty$, then we have, by assumption, that the r.h.s. tends to $0$ which yields the result.
\end{proof}

{\bf The linear case: Proof of Theorem \ref{alg-cont}. $V(x,u)=V(x)$.}
We recall that $V\in C^0(\R^d)$ is nonnegative and $V\neq 0$. Let us consider a complex-valued solution $u$ which satisfies
$$\int_{r_j\leq|x|\leq r_j+1}|u|^2\to 0, \quad \text{as}\; j \to+\infty.$$ If we denote $u=f+ig$, where $f$ and $g$ denote the real and imaginary parts of $u$, then $f$, $g$ are real-valued solutions which satisfy
$$\int_{r_j\leq|x|\leq r_j+1}g^2, \quad \int_{r_j\leq|x|\leq r_j+1}f^2\to 0, \quad \text{as}\; j\to+\infty.$$

Because $V\neq 0$ and $V$ is nonnegative, there exists a ball $B\subset\R^d$ where $V>0$. It follows from Proposition \ref{main-prop} that
$$\int_B Vf^2=\int_B Vg^2=0$$ which implies that $f,g\equiv 0$ on $B$ and, thus, $u\equiv0$ on $B$. 

Since $u$ satisfies the inequality $|\Delta u|\leq \|V\|_{L^\infty(\R^d)}|u|$, we may apply a classical unique continuation theorem, originally due to Carleman (cf. \cite{Car39}, see also section 1.5.2 in \cite{Ler19}), and conclude that $u\equiv 0$ on $\R^d$.\\

{\bf The non-linear case: Proof of Theorem \ref{alg-nonl}. $V(x,u)=V(u)$.} 
Let us assume that the zero set $Z_V=\{y\in\R \mid V(y)=0\}$ is an isolated point set.
If $|u(x)|\leq C$, then we get for any ball $B\subset\R^d$ $$\int_B V(u)u^2=0$$ and we conclude that $u=0$ unless $u(x)\in Z_V\setminus\{0\}$. Since $u$ is continuous and $Z_V$ is an isolated point set, it follows that $u\equiv const$. In particular, the assumption $$\int_{r_j\leq |x|\leq r_j+1}u^2\to 0$$ implies $u\equiv 0$ on $\R^d$.

\section{Proof of Corollary \ref{alg-quant}}\label{sec-coro}

We want to show that there exists $\epsilon>0$ s. t. for any $R\in\N$ we have
$$I(R):=\int_{R\leq|x|\leq R+1}|u|^2\geq \epsilon.$$

We argue by contradiction. Suppose that for any $\epsilon>0$ there exists $R\in\N$ s. t. $I(R) < \epsilon$.
Choose a sequence $(\epsilon_j) \subset \R_+$, $\epsilon_j\searrow 0$. For any $\epsilon_j>0$ there exists $R_j\in\N$ s. t. $I(R_j)<\epsilon_j\searrow 0$.

Note that the sequence $(R_j)\subset\N$ obtained in this way contains a subsequence $(R_{j_k})_k$ s. t. $R_{j_k}\nearrow+\infty$, as $k\to+\infty$. To see this, assume this is not the case. Hence, $(R_j)$ is bounded and, thus, can only take a finite number of values. In particular, there exists a constant subsequence $(R_{j_l})_l$ s. t. $R_{j_l}=R$ and $I(R_{j_l})=I(R)<\epsilon_{j_l}\to 0$, as $l\to+\infty$, and thus we conclude $I(R)=0$. This implies that $u\equiv 0$ on the annulus $R\leq|x|\leq R+1$ and, by Carleman's unique continuation theorem, we have $u\equiv 0$ on $\R^d$. However, we assumed $u$ to be non-trivial, so this gives a contradiction.

Upon reordering the subsequence $(R_{j_k})_k$ we obtain a sequence $(r_k)\subset\N$, $r_k\nearrow+\infty$, s. t. $I(r_k)\to 0$. Now we may apply Theorem \ref{alg-cont} to conclude that $u\equiv 0$ on $\R^d$ which is a contradiction to our assumption of non-triviality of $u$. Hence, we have proved the claim.

\section{Exterior domains}\label{sec-ext}

Let $\Omega$ be any unbounded subset with piecewise smooth boundary. Let us consider a real-valued solution $u$ of the Dirichlet problem
\beq
\Delta u=Vu, \quad u|_{\partial\Omega}=0.
\eeq

We may define the same smooth cutoff $v=\chi u$ as in the proof of Theorem \ref{alg-cont}. Then we have the identity
\beq
\int_{\Omega}|\nabla v|^2 = -\int_{\Omega} v \Delta v 
\eeq
where we note that the boundary terms vanish due to compact support as well as due to the boundary condition $u|_{\partial\Omega}=0$.

We have $$\Delta v=\chi \Delta u + 2\nabla\chi \cdot \nabla u + (\Delta\chi) u$$
and using the fact that $u$ solves $\Delta u=Vu$ we have
\beq\label{ext-id}
\int_{\Omega} V v^2 \leq -2\int_{\Omega}(\nabla\chi \cdot \nabla u) \chi u - \int_{\Omega}(\Delta\chi) u \chi u
\eeq

If we follow the same argument using the Caccioppoli inequality as in the proof of Theorem \ref{alg-cont} then we can prove that the l.h.s. of equation \eqref{ext-id} must vanish if there exists a sequence $(r_j)$, $r_j\nearrow+\infty$, s. t.
\beq
\int_{r_j \leq|x| \leq r_j+1,\,x\in\Omega}u(x)^2dx \to 0, \quad \text{as}\; r_j\to+\infty.
\eeq

\section{Auxiliary lemmata}

\subsection{Proof of Lemma \ref{Cacc}}\label{subsec-cacc}
Take $\eta_0\in C^\infty_c(\R^d)$ s. t. $0\leq \eta_0\leq1$ and $\eta_0=1$ on $B(0,1/4)$ and $\eta_0=0$ outside $B(0,1/2)$. In particular $|\nabla\eta_0|\leq c$ for some absolute constant $c$.

Take any $x_0\in\R^d$ and define $\eta=\eta_0(\cdot-x_0)$.

We have
$$\int_{\R^d}\Delta u \,\eta^2 u =\int_{\R^d} \eta^2 Vu^2$$
and the l.h.s. equals
$$-\int_{\R^d}\nabla u \cdot \nabla(\eta^2 u)=-\int_{\R^d}\nabla u \cdot (2\eta \nabla\eta u+\eta^2 \nabla u).$$

We rearrange this as
$$\int_{\R^d}\eta^2 |\nabla u|^2=-\int_{\R^d} \eta^2 Vu^2-\int_{\R^d}\nabla u \cdot (2\eta \nabla\eta u)$$

Hence, we have the inequality
$$\int_{\R^d}\eta^2 |\nabla u|^2 \leq \int_{\R^d} \eta^2 |V|u^2+2\int_{\R^d}\eta|\nabla u| |\nabla\eta| |u|$$

We can estimate (using the inequality $2ab\leq \tfrac{1}{2}a^2+2b^2$)
$$2\int_{\R^d}\eta|\nabla u| |\nabla\eta| |u| \leq \tfrac{1}{2}\int_{\R^d}\eta^2|\nabla u|^2+2\int_{\R^d}|\nabla \eta|^2u^2$$
and obtain
$$\tfrac{1}{2}\int_{\R^d}\eta^2 |\nabla u|^2 \leq \int_{\R^d} \eta^2 |V|u^2+2\int_{\R^d} |\nabla\eta|^2 u^2.$$

Now note that $V(x,u(x))\leq \sup_{x\in\R^d, |u|\leq C}V(x,u):= C_V<+\infty$, because $V(x,u)$ is in $L^\infty$ w.r.t. $x$ and continuous with respect to $u$. Moreover, $|u(x)|\leq C$ for any $x\in\R^d$.
This yields
$$\int_{B(x_0,1/4)}|\nabla u|^2\leq (2C_V+4c^2)\int_{B(x_0,1/2)}u^2.$$

\subsection{Construction of smooth cutoff}\label{subsec-cutoff}
Fix $\phi\in C^\infty_c(\R)$ with $\phi\geq0$, $\supp\phi\subset[0,1]$ and $\int_\R \phi=1$. We first construct the cutoff on $\R_+$:
\beq
\Phi_R(s):=\int_0^s\phi(t-R)dt.
\eeq
and we see that $\Phi_R$ is smooth, increasing and $0\leq \Phi_R\leq 1$. In particular, $\Phi_R(s)=0$ if $s\in[0,R]$, $0\leq\Phi_R(s)\leq 1$ for $s\in(R,R+1)$ and $\Phi_R=1$ for $s\geq R+1$.

We define a cutoff on a ball $B(0,R)\subset\R^d$ by
\beq
\chi(x)=1-\Phi_R(|x|).
\eeq
We have $\chi=1$ on $B(0,R)$, $0\leq\chi\leq 1$ on the annulus $B(0,R+1)\setminus B(0,R)$ and $\chi=0$ on $\R^d\setminus B(0,R+1)$.

We compute
$$\partial_{x_i}\chi(x)=-\Phi_R'(|x|)\frac{x_i}{|x|}$$
and since the Laplacian in radial coordinates is given by
$$\Delta_r=\frac{d^2}{dr^2}+\frac{(d-1)}{r}\frac{d}{dr}$$
and, thus, $|\nabla\chi|\leq \sup|\Phi_R'|=\sup|\phi|$ 
and since $\chi(x)=1-\Phi_R(r)$ with $r=|x|$, we have
$$\Delta\chi(x)=-\Delta_r\Phi_R(r)=-\Phi_R''(r)-\frac{(d-1)}{r}\Phi_R'(r)$$
and, thus, $$|\Delta\chi|\leq \sup|\Phi_R''|+\frac{d-1}{R}\sup|\Phi_R'|=\sup|\phi'|+\frac{d-1}{R}\sup|\phi|,$$ where we used that $r\in[R,R+1]$, because $\Phi_R',\Phi_R''$ are supported in the annulus $R\leq|x|\leq R+1$.


\begin{thebibliography}{99}

\bibitem{BouKen05}
J. Bourgain, C. Kenig, {\em On localization in the continuous Anderson-Bernoulli model in higher dimension.} Invent. Math. 161 (2005), 389--426.
\bibitem{Car39}
T. Carleman, {\em Sur un probl\`eme d'unicit\'e pour les syst\`emes d'\'equations aux deriv\'ees partielles \`a deux variables ind\'ependantes}, Arkiv f\"{o}r Matematik, Astronomi och Fysik, 26B, no. 17 (1939), pp. 1-9.
\bibitem{PiDa24}
U. Das, Y. Pinchover, {\em The Landis conjecture via Liouville comparison principle and criticality theory.} arXiv:2405.11695.
\bibitem{PiDaK25}
U. Das, M. Keller, Y. Pinchover, {\em On Landis’ Conjecture for Positive Schrödinger Operators on Graphs.} IMRN 2025 (2025), no. 12, pp. 1-20.
\bibitem{DKW19}
B. Davey, C. Kenig, J.-N. Wang, {\em On Landis' conjecture in the plane when the potential has an exponentially decaying negative part.} Algebra i Analiz, 31:2 (2019), 204--226; St. Petersburg Math. J., 31:2 (2019), 337--353.
\bibitem{Lan53}
E. M. Landis, {\em Some questions in the qualitative theory of second-order elliptic equations.} Uspekhi Math. Nauk 8 (1953), no. 1 (49), pp. 8--31.
\bibitem{Lan71}
E. M. Landis, {\em Second-order equations of elliptic and parabolic type.} Nauka, Moscow, 1971.
\bibitem{Ler19}
N. Lerner, {\em Carleman Inequalities: An Introduction and More}, Grundlehren der Mathematischen Wissenschaften 353, Springer, 2019.
\bibitem{LBSou23}
K. Le Balc'h, and D. A. Souza, {\em Quantitative unique continuation for real-valued solutions to second order elliptic equations in the plane.} arXiv: 2401.00441.
\bibitem{Me92}
V. Z. Meshkov, {\em On the possible rate of decay at infinity of solutions to second order partial differential equations.} Math. USSR Sb. 72 (1992), 343--360.
\bibitem{KSW15}
C. Kenig, L. Silvestre, J.-N. Wang, {\em On Landis' conjecture in the plane} Comm. PDE 40 (2015), 766--789.
\bibitem{LMNN24}
A. Logunov, E. Malinnikova, N. Nadirashvilli, F. Nazarov, {\em The Landis conjecture on exponential decay}, Invent. Math. 241 (2025), pp. 465--508.
 \bibitem{SiSou21}
 B. Sirakov, and P. Souplet, {\em The Vazquez maximum principle and the Landis conjecture for elliptic PDE
with unbounded coefficients.} Adv. Math. 387 (2021), 107838.
\end{thebibliography}
\end{document}